\documentclass[12pt,reqno]{article}

\usepackage[usenames]{color}
\usepackage{amssymb}
\usepackage{amsmath}
\usepackage{amsthm}
\usepackage{amsfonts}
\usepackage{amscd}
\usepackage{graphicx}
\usepackage{subcaption}
\usepackage[shortlabels]{enumitem}

\usepackage[colorlinks=true,
linkcolor=webgreen,
filecolor=webbrown,
citecolor=webgreen]{hyperref}

\definecolor{webgreen}{rgb}{0,.5,0}
\definecolor{webbrown}{rgb}{.6,0,0}

\usepackage{color}
\usepackage{fullpage}
\usepackage{float}
\usepackage{graphics}
\usepackage{latexsym}
\usepackage{epsf}

\DeclareMathOperator{\NR}{NR}
\DeclareMathOperator{\NC}{NC}

\theoremstyle{plain}
\newtheorem{thm}{Theorem}[section]
\newtheorem{lem}[thm]{Lemma}

\theoremstyle{definition}
\newtheorem{defi}[thm]{Definition}
\newtheorem{conj}[thm]{Conjecture}

\def\modd#1 #2{#1\ \mbox{\rm (mod}\ #2\mbox{\rm )}}

\begin{document}

\begin{center}
\vskip 1cm{\LARGE\bf
Modeling Random Walks to Infinity on Primes in $\mathbb{Z}[\sqrt{2}]$
}
\vskip 1cm
\large
Bencheng Li\\
Department of Mathematics\\
University of Michigan\\
Ann Arbor, MI 48105\\
USA \\
\href{mailto:benchenl@umich.edu}{\tt benchenl@umich.edu} \\
\ \\
Steven J. Miller\\
Department of Mathematics and Statistics\\
Williams College\\
Williamstown, MA 01267\\
USA \\
\href{mailto:sjm1@williams.edu}{\tt sjm1@williams.edu} \\
\ \\
Tudor Popescu\\
Department of Mathematics\\
Brandeis University\\
Waltham, MA 02453\\
USA \\
\href{mailto:tudorpopescu@brandeis.edu}{\tt tudorpopescu@brandeis.edu} \\
\ \\
Daniel Sarnecki\\
Department of Mathematics\\
Cornell University\\
Ithaca, NY 14850\\
USA \\
\href{mailto:dbs263@cornell.edu}{\tt dbs263@cornell.edu} \\
\ \\
Nawapan Wattanawanichkul\\
Department of Mathematics\\
University of Illinois Urbana-Champaign\\
Urbana, IL 61801\\
USA\\
\href{mailto:nawapan2@illinois.edu}{\tt nawapan2@illinois.edu} \\
\ \\
\end{center}

\vskip .2 in

\begin{abstract}
An interesting question, known as the Gaussian moat problem, asks whether it is possible to walk to infinity on Gaussian primes with steps of bounded length. Our work examines a similar situation in the real quadratic integer ring $\mathbb{Z}[\sqrt{2}]$ whose primes cluster near the asymptotes $y = \pm x/\sqrt{2}$ as compared to Gaussian primes, which cluster near the origin. We construct a probabilistic model of primes in $\mathbb{Z}[\sqrt{2}]$ by applying the prime number theorem and a combinatorial theorem for counting the number of lattice points whose absolute values of their norms are at most $r^2$. We then prove that it is impossible to walk to infinity if the walk remains within some bounded distance from the asymptotes. Lastly, we perform a few moat calculations to show that the longest walk is likely to stay close to the asymptotes; hence, we conjecture that there is no walk to infinity on $\mathbb{Z}[\sqrt{2}]$ primes with steps of bounded length.
\end{abstract}

\section{Introduction}\label{sec:introduction}

\subsection{Background}\label{subsec:background} It is known that one cannot walk to infinity along real prime numbers with steps of bounded length, i.e., there is no finite number $N$ such that there exists an infinite sequence of increasing real primes $p_1, p_2, \ldots$ where $p_{i+1}-p_{i}\leq N$ for any $i\in \mathbb{N}$. One can prove this by using the \textit{primorial} $p\#$, which is the product of all primes $p_i$ less than or equal to $p$, i.e.,
$$p\# :=  \prod_{p_i  \text{ prime, } p_i \leq p} p_i.$$
For any real prime $p$, consider the sequence $$(p\#+j)_{j=2, 3, \ldots, p}.$$
All numbers in this sequence are composite, so we have a gap of length at least $p-1$ between $p$ and the next prime. Thus, it is impossible to walk to infinity along primes taking steps of bounded length because there are infinitely many primes. 

This classical problem is clearly one-dimensional as we only need to find an arbitrarily large gap on the number line. However, one can ask the more flavorful question of whether there exists a prime walk to infinity in \textit{two dimensions}, for example, in the integer ring of a quadratic field \cite[p.\ 368]{D1994}.

\begin{defi}\label{def:rings}
For a square-free integer $d$, the \textit{quadratic integer ring of a quadratic field} $\mathbb{Q}(\sqrt{d})$, denoted by $\mathbb{Z}[\sqrt{d}]$, is the set
\begin{enumerate}[(i)]
\item $\{a + b\left(\frac{1 + \sqrt{d}}{2}\right) \mid a, b \in \mathbb{Z}\}$ if $d \equiv 1$ (mod $4$), \text{or}
\item $\{a + b\sqrt{d} \mid a, b \in \mathbb{Z}\}$ if $d \equiv 2,3$ (mod $4$).
\end{enumerate}
\end{defi}
We note that for any $d$, the \textit{norm} of an element $a+b\sqrt{d}$ in $\mathbb{Z}[\sqrt{d}]$ is defined as $a^{2} - b^2d$. We call an element in $\mathbb{Z}[\sqrt{d}]$ a \textit{unit} if its norm is $\pm 1$, and two numbers in $\mathbb{Z}[\sqrt{d}]$ are \textit{associates} if one is a multiple of the other by a unit. Lastly, in a unique factorization domain (UFD), for example, $\mathbb{Z}[i]$ and  $\mathbb{Z}[\sqrt{2}]$,  a \textit{prime} is defined as a non-zero, non-unit number that is divisible only by its associates and the units of such UFD \cite[p.\ 268]{HW}. 

In particular, the \textit{ring of Gaussian integers} $\mathbb{Z}[i] := \{a+bi \mid a,b\in \mathbb{Z}\}$ is a special case of quadratic integer rings when $d = -1$, and primes in $\mathbb{Z}[i]$ are simply called \textit{Gaussian primes}. With the notion of primes established, one can ask the earlier question regarding the existence of a prime walk to infinity in this setting. This problem is also known as the \textit{Gaussian moat problem}. 

\subsection{Gaussian moat problem and our motivation}\label{subsec:quadraticring}

According to Guy \cite[Problem A.16]{G} and Gethner and Stark \cite{GS}, Basil Gordon first posed the \textit{Gaussian moat problem} in 1962, asking whether one can walk to infinity starting at the origin, and thereafter stepping only on primes in $\mathbb{Z}[i]$ with steps of bounded length. Equivalently, the problem discusses the existence of a \textit{$k$-moat} for any finite real $k$, where a $k$-moat is a region of composite Gaussian integers that encloses the origin with width at least $k$. A more general version of this problem also considers any starting point besides the origin \cite{GS}.

\begin{figure}[ht]
\centering
\begin{subfigure}[t]{0.38\textwidth}
  \centering
  \includegraphics[width=\linewidth]{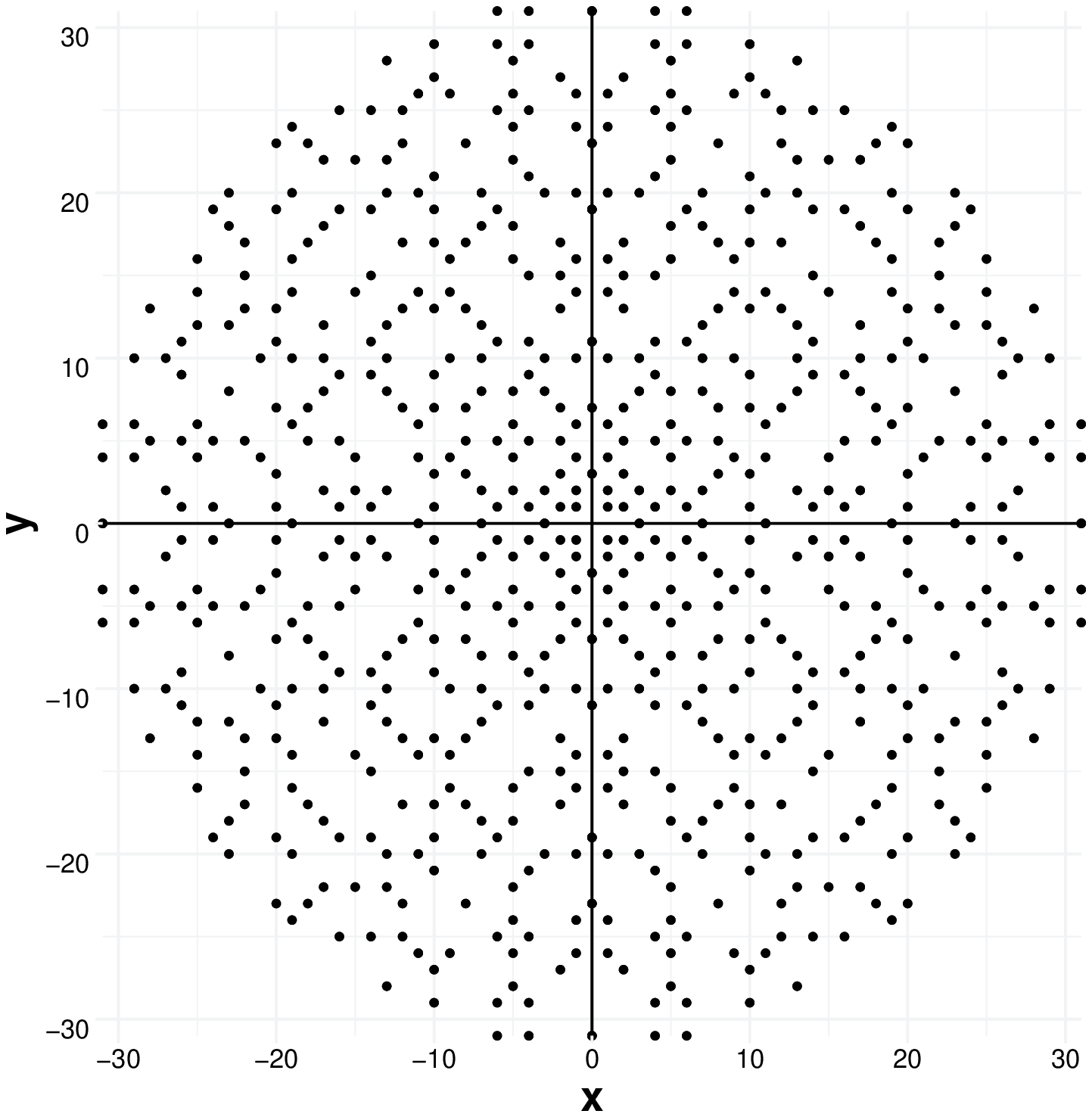}
  \caption{}
  \label{fig:Gplot30}
 \end{subfigure}
\begin{subfigure}[t]{0.38\textwidth}
  \centering
  \includegraphics[width=\linewidth]{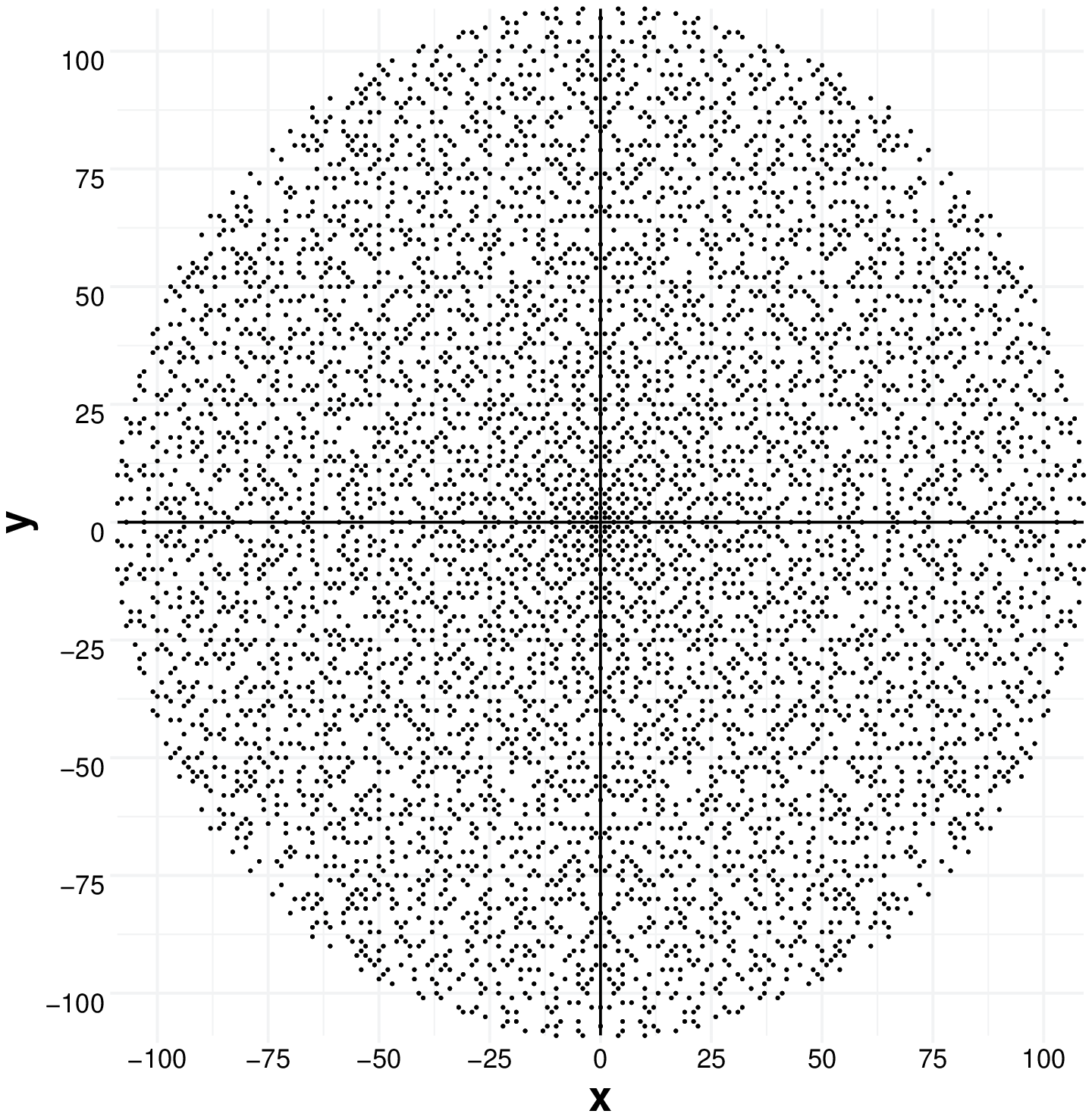}
  \caption{}
  \label{fig:Gplot100}
\end{subfigure}
\caption{Gaussian primes in disks of radius $30$ (Figure \ref{fig:Gplot30}) and $100$ (Figure \ref{fig:Gplot100}).}
\end{figure}

To simplify the Gaussian moat problem, one can exploit the \textit{$8$-fold symmetry} of the geometric structure of Gaussian primes as manifested in Figures \ref{fig:Gplot30} and \ref{fig:Gplot100}. Such symmetry can be explained by the following characterization of Gaussian primes \cite[p.\ 378]{D1994}.

\begin{thm}\label{thm:Gaussianprime}
An element $a+bi \in \mathbb{Z}[i]$ is a Gaussian prime if and only if the element satisfies one of the following requirements up to associates:
\begin{enumerate}[(i)]
    \item\label{thm:Gauss1} $a+bi$ is $1+i$,
    \item\label{thm:Gauss2} $a+bi$ is such that $a^2+b^2$ is a real prime, and $a^2+b^2 \equiv \modd {1} {4}$, or
    \item\label{thm:Gauss3} $a+bi$ is such that $b = 0$, $|a|$ is a real prime, and $a \equiv \modd {3} {4}$.
\end{enumerate}
\end{thm}

Since the norm of $a+bi$ is defined as $a^2+b^2$, any element $a+bi$ is a unit if $a^2+b^2 = 1$. This gives us exactly four units in $\mathbb{Z}[i]$, namely, $\pm 1$ and $\pm i$. Thus, for each prime in $\mathbb{Z}[i]$, its other three associates are also primes. Moreover, it is known that any real prime $p$ congruent to 1 modulo 4 can be expressed as a unique sum $a^2+b^2$ \cite[p.\ 284]{HW}. Thus, by Theorem \ref{thm:Gaussianprime} \ref{thm:Gauss2}, each real prime $p \equiv 1$ (mod $4$) yields eight Gaussian primes: $a+bi$, $a-bi$, and their six associates, which implies the $8$-fold symmetry. Then one can reduce the problem to finding a strip of some width at least $k$ sweeping across the first octant. 

Much progress towards the Gaussian moat problem has been made over decades. Jordan and Rabung \cite{JR}, in 1970, constructed a $\sqrt{10}$-moat when starting at the origin while Gethner et al.\ \cite{GWW} proved the existence of $4$, $\sqrt{18}$, and $\sqrt{26}$-moats in 1998. The latter result implies that steps of length 5 are not sufficient for a walk to infinity. Moreover, in 2004, Tsuchimura \cite{T} showed the existence of a $6$-moat. While the problem remains open, there is a strong inclination from the literature \cite{GS, GWW} that the answer to the Gaussian moat problem is \textit{no}.

As mentioned earlier, the Gaussian moat problem is just a variation of the prime walk to infinity problem. Hence, in this paper, we would like to shift our focus to another quadratic ring of integers, namely, $\mathbb{Z}[\sqrt{2}]$. In other words, we attempt to answer \textit{whether it is possible to walk to infinity on primes in $\mathbb{Z}[\sqrt{2}]$ with steps of bounded length}. To begin with, we recall the following characterization of the primes in $\mathbb{Z}[\sqrt{2}]$ \cite[p.\ 378]{D1994}.

\begin{thm}\label{thm:sqrt2primes}
An element $a+b\sqrt{2}\in \mathbb{Z}[\sqrt{2}]$ is a prime if and only if the element satisfies one of the following requirements up to associates:
\begin{enumerate}[(i)]
    \item \label{thm:form1} $a+b\sqrt{2}$ is $\sqrt{2}$,
    \item \label{thm:form2} $a+b\sqrt{2}$ is such that $|a^2-2b^2|$ is a real prime,  and $a^2-2b^2 \equiv \modd{1, 7} {8}$, or
    \item \label{thm:form3} $a+b\sqrt{2}$ is such that $b=0$, $a$ is a real prime, and $a \equiv \modd {3, 5} {8}$.
\end{enumerate}
\end{thm}
It is known that $1+\sqrt{2}$ is the \textit{fundamental unit of $\mathbb{Z}[\sqrt{2}]$}, meaning that all units of $\mathbb{Z}[\sqrt{2}]$ are of the form $\pm(1+\sqrt{2})^k$, where $k \in \mathbb{Z}$ \cite[p.\ 270]{HW}. Thus, for each prime $a+b\sqrt{2} \in \mathbb{Z}[\sqrt{2}]$, its associates $\pm(a+b\sqrt{2})(1+\sqrt{2})^k$ are also primes. In addition, using Theorem \ref{thm:sqrt2primes} \ref{thm:form2}, one can show that every prime number $p \equiv 1, 7$ (mod $8$) can be expressed as $a^2 - 2b^2$, for some $a, b \in \mathbb{Z}$ by using Thue's lemma \cite[p.\ 43]{S}.

\begin{lem}[Thue's lemma]\label{lem:thue's} Let $a, n > 1$ be integers with $\gcd(a, n) = 1$. Then there exist $x, y \in \mathbb{Z}$ with $$0 < |x|, |y| < \sqrt{n}$$ and $$x \equiv ay \pmod n.$$
\end{lem}

\begin{proof}
Let $r = \lfloor \sqrt{n} \rfloor$. Then $r^2 \le n < (r+1)^2$. Note that the number of pairs $(x, y)$ with $0 \le x, y \le r$ is $(r + 1)^2 > n$. Therefore, by the pigeonhole principle, there exist $(x_1, y_1)$ and $(x_2, y_2)$ as above such that $$x_1 - ay_1 \equiv x_2 - ay_2 \pmod n.$$

In particular, we obtain that $x_1 - x_2 \equiv a(y_1 - y_2)$ (mod $n$). Setting $x = x_1 - x_2$ and $y = y_1 - y_2$, we get the condition in the above statement, and we only have to prove that $0 < |x|, |y| < \sqrt{n}$.

Since $\gcd(a, n) = 1$, we have that if $x = 0$, then $y = 0$. However, if $x = y = 0$, then $(x_1, y_1) = (x_2, y_2)$, contradicting the fact that they are different pairs. Since $0 \le x_i, y_i \le r$ for $i \in \{1,2\}$, we conclude that $0 < |x|, |y| < r \le \sqrt{n}$ and $x \equiv ay$ (mod $n$).
\end{proof}

With the above lemma established, we now obtain the following theorem.
\begin{thm}
Let $p$ be a prime number. Then $p$ can be written as $a^2 - 2b^2, a, b \in \mathbb{Z}$ if and only if $p \equiv \modd {1, 7} {8}$.
\end{thm}

\begin{proof} If $p = a^2 - 2b^2$ where $a, b \in \mathbb{Z}$, then $a^2 \equiv 2b^2$ (mod $p$). Hence, $2$ is a quadratic residue modulo $p$, but this is equivalent to $p \equiv 1,7$ (mod $8$).

On the other hand, if $p \equiv 1, 7$ (mod $8$), then $2$ is a quadratic residue modulo $p$, so there exists $c$ such that $c^2 \equiv 2$ (mod $p$). Using Lemma \ref{lem:thue's}, there exist $a, b \in \mathbb{Z}$ with $0 < |a|, |b| < \sqrt{p}$ and $a \equiv bc$ (mod $p$). We then square the previous expression and obtain
$$a^2 \equiv b^2c^2 \equiv 2b^2 \pmod p  \ \Rightarrow \ a^2 - 2b^2 \equiv 0 \pmod p.$$
Since $0 < a, b < \sqrt{p}$ and $p|(a^2 - 2b^2)$, then $a^2 - 2b^2 = \pm p$. If $a^2 - 2b^2 = p$, we have proved the theorem. Otherwise, $-p = a^2 - 2b^2$, and note that
$$p = 2b^2 - a^2 = (a + 2b)^2 - 2(a + b)^2.$$ \end{proof}

Therefore, by using Theorem \ref{thm:sqrt2primes} \ref{thm:form2} and the fundamental unit $1+\sqrt{2}$, each prime $p \equiv 1, 7$ (mod $8$) contributes infinitely many $\mathbb{Z}[\sqrt{2}]$ primes of the form $\pm(x\pm y\sqrt{2})(1+\sqrt{2})^{2k}$, where $p = x^2-2y^2$ and $k$ is any integer. However, there are only four associates with the same Euclidean norm, namely, $\pm(x\pm y\sqrt{2})$. This implies the \textit{$4$-fold symmetry} of primes in $\mathbb{Z}[\sqrt{2}]$ as contrast to the 8-fold symmetry of $\mathbb{Z}[i]$. Hence, it is sufficient to study walks on $\mathbb{Z}[\sqrt{2}]$ primes in the first quadrant. In Figures \ref{fig:plot50} and \ref{fig:plot200}, we notice that primes in $\mathbb{Z}[\sqrt{2}]$ tend to cluster near the asymptotes $y = \pm x/\sqrt{2}$, whereas Gaussian primes only cluster near the origin. Because of this distinct geometric structure of primes in $\mathbb{Z}[\sqrt{2}]$, if there exists an unbounded walk in $\mathbb{Z}[\sqrt{2}]$ that steps on the primes with steps of bounded length, then the walk is expected to stay close to the asymptotes.

\begin{figure}[ht]
\centering
\begin{subfigure}[t]{0.40\textwidth}
  \centering
  \includegraphics[width=0.9\linewidth]{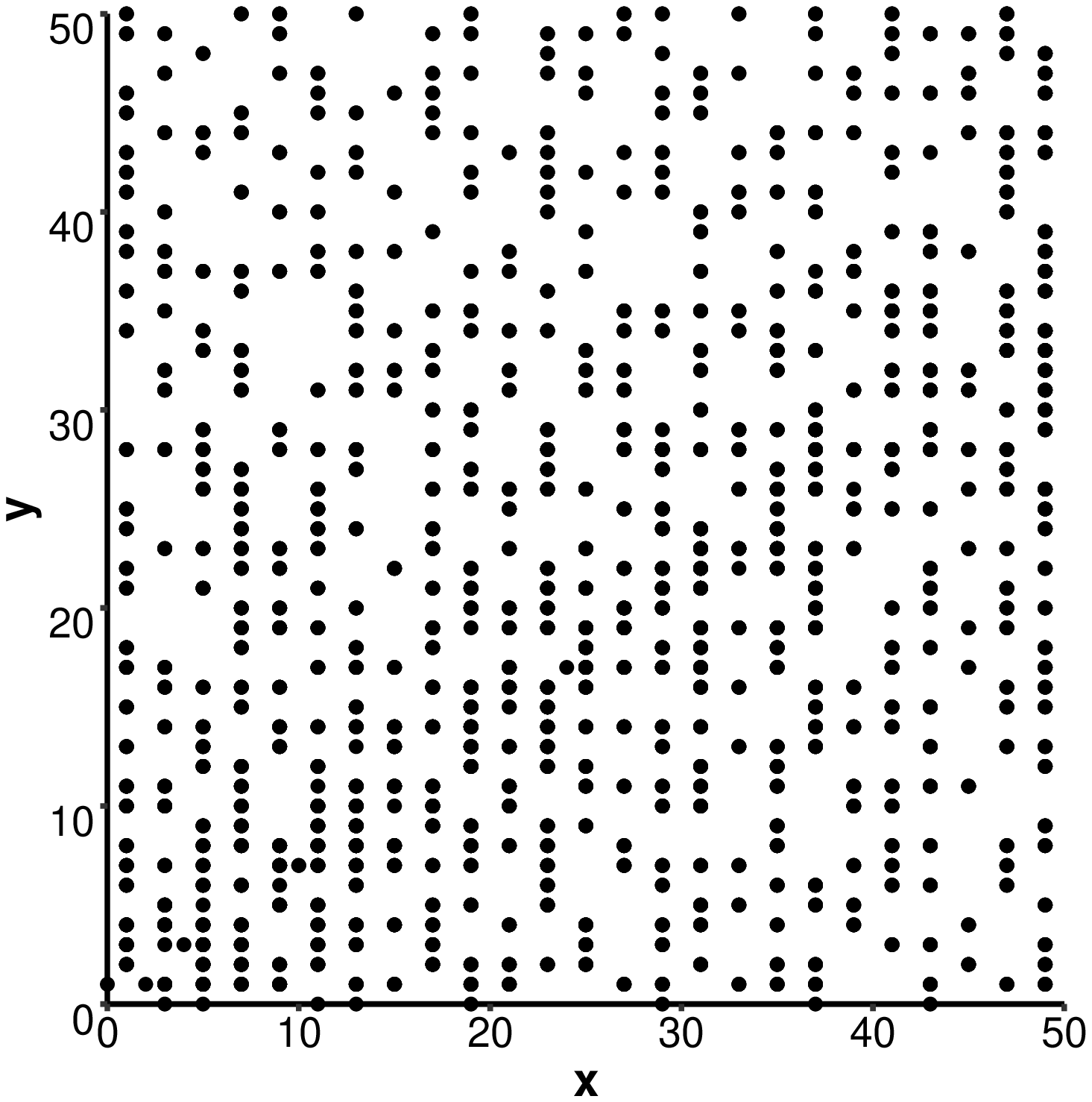}
  \caption{}
  \label{fig:plot50}
 \end{subfigure}
\begin{subfigure}[t]{0.40\textwidth}
  \centering
  \includegraphics[width=0.9\linewidth]{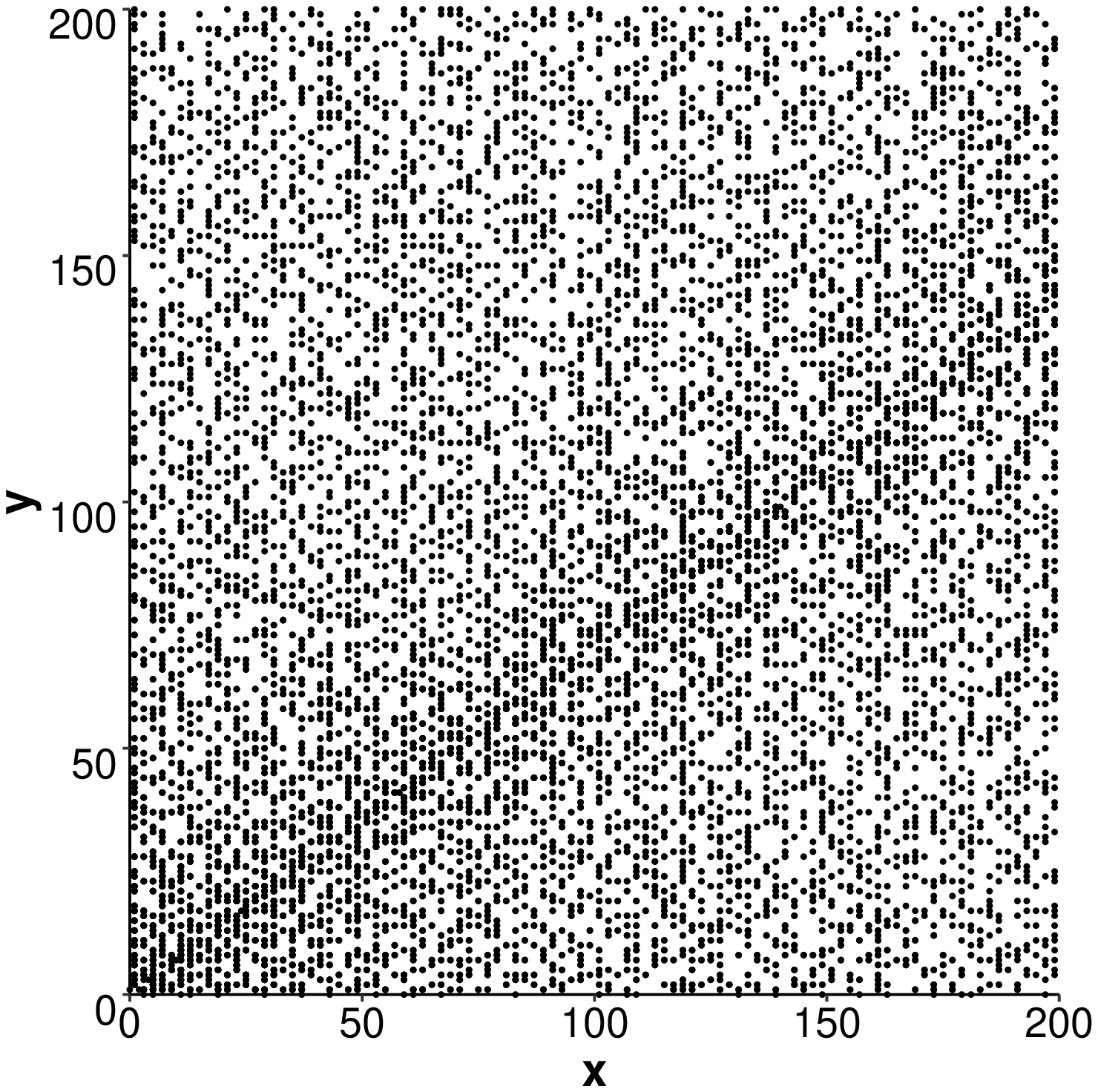}
  \caption{}
  \label{fig:plot200}
\end{subfigure}
\caption{Primes in $\mathbb{Z}[\sqrt{2}]$ where $0 \le x,y \le 50$ (Figure \ref{fig:plot50}) and $0 \le x,y \le 200$ (Figure \ref{fig:plot200}).}
\end{figure}

\subsection{Results}
We observe in Figure \ref{fig:comparison} that the number of primes in the disk centered at the origin with radius $n < 60$ in $\mathbb{Z}[\sqrt{2}]$ is greater than that of $\mathbb{Z}[i]$. The observation leads us to Theorem \ref{thm:prob_prime_in_NR} in Section \ref{section:2Modeling}, in which we show that the density of primes in $\mathbb{Z}[\sqrt{2}]$ is higher than that in $\mathbb{Z}[i]$. We also deduce from Theorem \ref{thm:prob_prime_in_NR} and Figures \ref{fig:plot50} and \ref{fig:plot200} that a prime walk in $\mathbb{Z}[\sqrt{2}]$ in the first quadrant is more likely to exist near the asymptote $y = x/\sqrt{2}$. In Section \ref{sec:proof_of_main_results}, we then present the main theorem of this paper, Theorem \ref{thm:main_theorem}, which states that it is impossible to perform a walk to infinity taking steps of bounded length on primes in $\mathbb{Z}[\sqrt{2}]$ if the walk remains within some bounded distance from the asymptote $y = x/\sqrt{2}$. Lastly, in Section \ref{sec:visualizing_prime_walks}, we present some evidence that the longest walk possible must stay close to the asymptote. We are therefore led to Conjecture \ref{conj:main}, which states that it is impossible to have such a walk to infinity in $\mathbb{Z}[\sqrt{2}]$.

\begin{figure}[ht]
\centering
\includegraphics[width=2.4in]{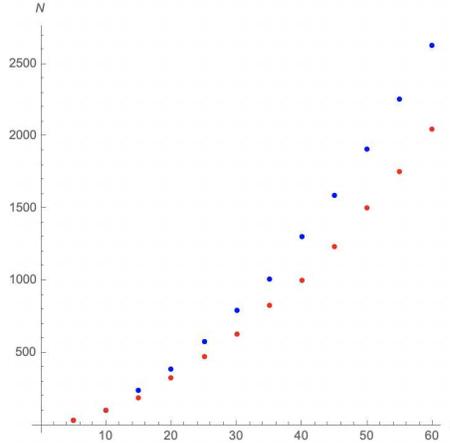}
\caption{The number of primes ($N$) in $\mathbb{Z}[i]$ (red, bottom) and  $\mathbb{Z}[\sqrt{2}]$ (blue, top) within the disk centered at the origin with radius $n$.}
\label{fig:comparison}
\end{figure}

\section{Modeling primes in \texorpdfstring{$\mathbb{Z}[\sqrt{2}]$}{Z[/2]}}\label{section:2Modeling}

\subsection{Density of Gaussian primes}

Our goal in this section is to obtain an asymptotic count for the density of primes in $\mathbb{Z}[\sqrt{2}]$. To do so, we adapt Loh's approximation of Gaussian primes in the disk of radius $r$ centered at the origin \cite[p.\ 143]{L}, which relies on the \textit{prime number theorem} \cite[p.\ 10]{HW} and \textit{Dirichlet's theorem on arithmetic progressions} \cite[pp.\ 542--545]{N}.

\begin{thm}[Prime number theorem, PNT]\label{thm:PNT} Let $\pi(x)$ be the number of primes not exceeding a real number $x$. Then
$$\pi(x) \sim \frac{x}{\log x}.$$
\end{thm}

We note that here, as well as throughout the paper, $\log$ refers to the natural logarithm, and $f(x) \sim g(x)$ means that $\lim_{x\to \infty} f(x)/g(x) = 1.$

\begin{defi}[Dirichlet density] For two sets of positive integers $S \subseteq T$, with $\sum_{n\in T} n^{-1}$ divergent, the \textit{Dirichlet density of $S$ in $T$}, denoted by $d(S,T)$, exists if and only if
\begin{equation*}
    \lim_{s \to 1^{+}} \inf \frac{\sum_{n \in S} n^{-s}}{\sum_{n \in T} n^{-s}}  =   d(S,T)  =  \lim_{s \to 1^{+}} \sup \frac{\sum_{n \in S} n^{-s}}{\sum_{n \in T} n^{-s}}.
\end{equation*}
\end{defi}

\begin{defi}\label{def:P,a,m} We let $\mathbb{P}$ be the set of positive primes and $\mathbb{P}(a;m)$ the set of positive primes $p$ such that $p\equiv a$ (mod $m$).
\end{defi}

\begin{thm}[Primes in arithmetic progressions]\label{thm:dirichlet_density}
Suppose $a,m \in \mathbb{Z}$, with $\gcd(a,m) =1$.  Then $d(\mathbb{P}(a;m), \mathbb{P}) = 1/\varphi(m)$, where $\varphi$ is Euler's totient function.
\end{thm}
Using the above tools, Loh \cite[p.\ 143]{L} has obtained the following theorem.

\begin{thm} \label{thm:Gaussian_disk}
The number of Gaussian primes contained within the disk of sufficiently large radius $r$ about the origin is asymptotic to  $$\frac{2r^2}{\log r}.$$
\end{thm}
\begin{proof}

Theorem \ref{thm:dirichlet_density} implies that the densities of primes of the forms $4k+1$ and $4k+3$ are the same since $d(\mathbb{P}(1;4), \mathbb{P}) = d(\mathbb{P}(3;4), \mathbb{P}) = 1/\varphi(4) = 1/2$. As in Theorem \ref{thm:Gaussianprime}, all the Gaussian primes can be derived from real primes as follows: \begin{enumerate}[(i)]
    \item\label{obs:2} the real prime 2 gives us four Gaussian primes: $1+i, 1-i, -1+i$, and $-1-i$,
    \item\label{obs:4k+3} any real prime $p=4k+3$ supplies four Gaussian primes: $p, -p, pi$, and $-pi$,
    \item\label{obs:4k+1} any real prime $p=4k+1$ supplies eight Gaussian primes as $p$ can be written as a unique sum $a^2+b^2 = (a+bi)(a-bi)$, i.e.,  each $p$ gives $a+bi, -(a+bi), (a+bi)i, -(a+bi)i, a-bi, -(a-bi), (a-bi)i$, and $-(a-bi)i$.
\end{enumerate}

Now, consider the disk centered at the origin with radius $r \ge 1$. From \ref{obs:4k+3}, any real $(4k+3)$-prime supplies four Gaussian primes on the real and imaginary axes. We derive from the PNT and Theorem \ref{thm:dirichlet_density} that there are asymptotically $r/(\varphi(4)\log r)$ $(4k+3)$-primes, so the number of Gaussian primes lying on the axes within the disk is asymptotic to
\begin{equation}{4\cdot \frac{r}{\varphi(4)\log r} = \frac{2r}{\log r}}\label{eqn:Gaussian_primes_axes}. \end{equation}

Next, consider the disk excluding the axes. For positive integers $1,2,\ldots, r^2$, there are asymptotically $r^2/\log r^2$ primes. Note that we consider the integers from 1 to $r^2$ since each lattice point in the disk has the norm of at most $r^2$. By Theorem \ref{thm:dirichlet_density}, asymptotically half of them are primes of the form $4k+1$. From \ref{obs:4k+1}, the number of Gaussian primes supplied by these $(4k+1)$-primes is asymptotic to
\begin{equation}{8\cdot\frac{r^2}{2\log r^2} =  \frac{2r^2}{\log r}}.\label{eqn:Gaussian_primes_interior}\end{equation} Combining the results from \ref{obs:2}, \eqref{eqn:Gaussian_primes_axes}, and \eqref{eqn:Gaussian_primes_interior} gives the desired estimate.
\end{proof}

Thus, the density of Gaussian primes in the disk of radius $r$ centered at the origin can be obtained by dividing the estimate in Theorem \ref{thm:Gaussian_disk} by the number of lattice points in the disk, which is known to be $\pi r^2 + O(r)$ \cite[p.\ 67]{H}.

\begin{thm} \label{thm:Gaussian density}
The density of Gaussian primes in the disk of radius $r$ centered at the origin is asymptotic to   $$\frac{2}{\pi\log r},$$ i.e., the probability that a Gaussian integer in the disk is prime is asymptotically $2/(\pi\log r)$.
\end{thm}

\subsection{Density of primes in  \texorpdfstring{$\mathbb{Z}[\sqrt{2}]$}{Z[/2]}}

In generalizing the PNT to the Gaussian primes, we consider primes whose norms lie within an interval. In particular, we estimate the number of Gaussian primes of the form $a+bi$ such that $a^2+b^2 \le r^2$. We follow this line of thinking in generalizing further to primes in $\mathbb{Z}[\sqrt{2}]$, where we take the absolute value of the norm as the norm can be negative. Here, we estimate the number of primes of the form $a+b\sqrt{2}$ such that \begin{equation}\label{eqn:norminequality}
|a^2 - 2b^2| \leq r^2. \end{equation}

\begin{figure}[ht]
  \begin{subfigure}[t]{0.48\textwidth}
  \centering
    \includegraphics[width=0.9\textwidth]{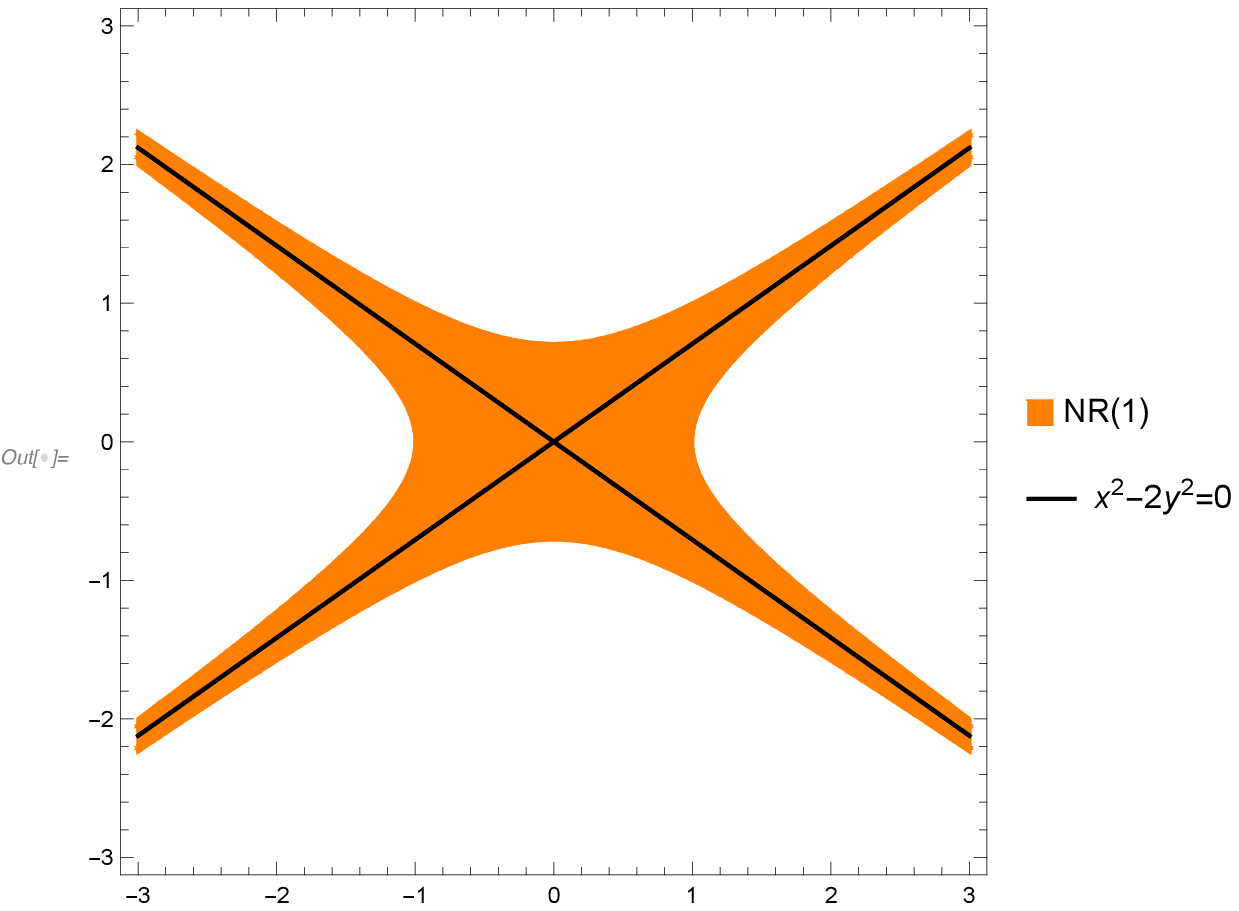}
    \caption{}
    \label{fig:NR1}
  \end{subfigure}
  \begin{subfigure}[t]{0.48\textwidth}
  \centering
    \includegraphics[width=0.9\textwidth]{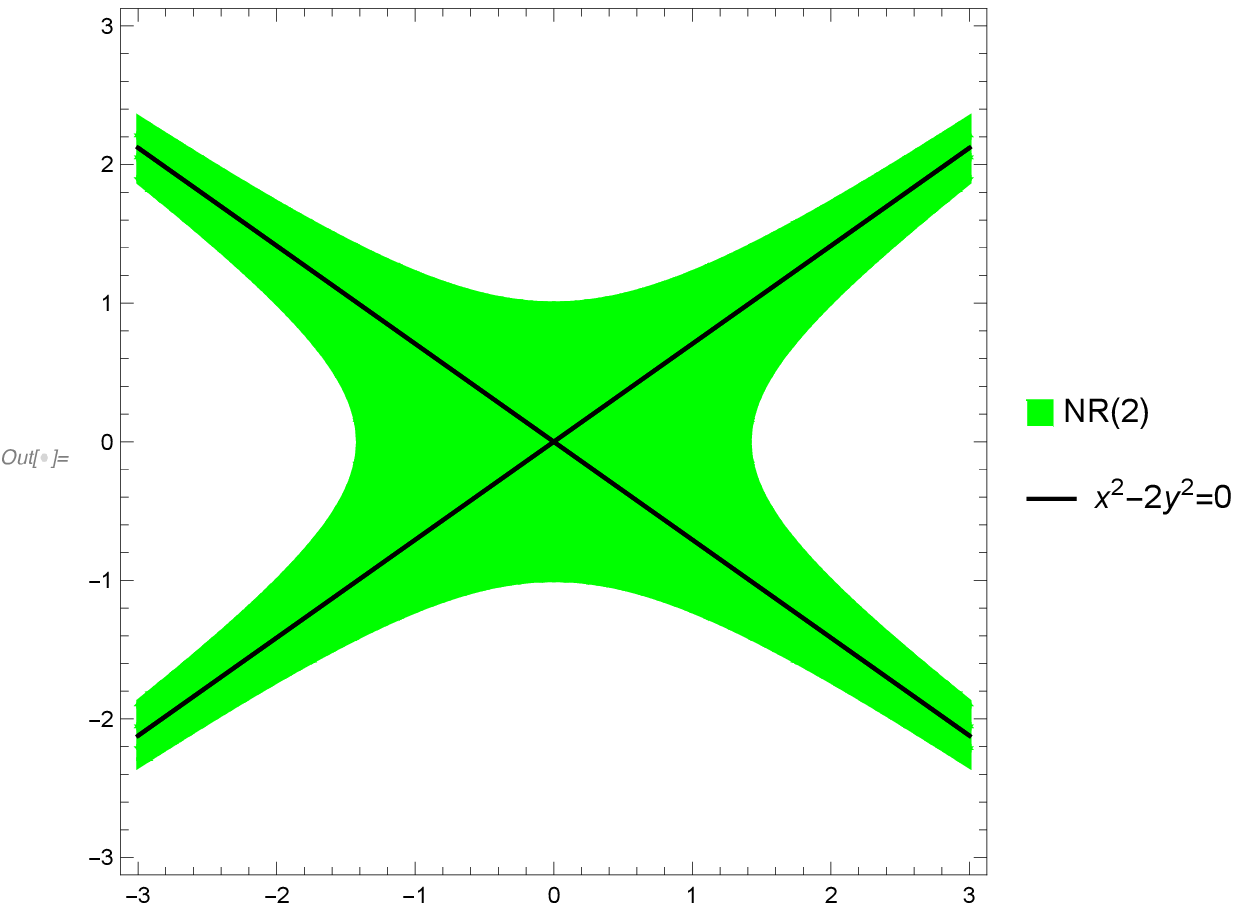}
    \caption{}
    \label{fig:NR2}
  \end{subfigure}
  \caption{Norm regions $\NR(1)$ (orange; Figure \ref{fig:NR1}) and $\NR(2)$ (green; Figure \ref{fig:NR2}), enclosed by $|x^2-2y^2|=1$ and $|x^2-2y^2|=2$, respectively, and the asymptotes $y= \pm x/\sqrt{2}$, i.e., the graph $x^2-2y^2 =0$ (black).}
\end{figure}

\begin{defi}\label{def:norm-region}
    Let $\NR(r^2)$ be the \textit{norm region} defined by $|x^2 - 2y^2| \leq r^2$ where $x,y \in \mathbb{R}$.
\end{defi}

Clearly, $\NR(r^2)$ contains all the points $(a,b) \in \mathbb{Z}[\sqrt{2}]$ that satisfy \eqref{eqn:norminequality}, so our goal is then to estimate the number of primes within $\NR(r^2)$. However, unlike the disk norm region $x^2 +y^2 \le r^2$ that we use for counting the Gaussian primes, $\NR(r^2)$ is unbounded with infinite area as shown in Figures \ref{fig:NR1} and \ref{fig:NR2}. Thus, there might be infinitely many primes within $\NR(r^2)$. To verify this statement and for convenience, we define the following notation:

\begin{defi} \label{def:k-norm_curve}
For any $k \in \mathbb{Z}$, let $\NC(k)$ denote the \textit{$k$-norm curve} or the graph of the equation $x^2-2y^2=k$.
\end{defi}

\begin{lem} \label{lem:infinite_primes}
For any $r^2 \ge 2$, there exist infinitely many primes within $\NR(r^2)$.
\end{lem}
\begin{proof}
Notice that the $k$-norm curve that is closest to the origin and contains at least one prime is $\NC(-2)$ because $\sqrt{2}$ and $-\sqrt{2}$ lie on the curve. We know that, for every prime $x+y\sqrt{2}$ in $\mathbb{Z}[\sqrt{2}]$, its associates of the form
$\pm(x+y\sqrt{2})(1+\sqrt{2})^{2n}$, where $n \in \mathbb{Z}$,
are also primes on the same curve. In particular, there are infinitely many primes on $\NC(-2)$, so any norm region $\NR(r^2)$ where $r^2 \ge 2$ contains infinitely many primes.
\end{proof}

According to Lemma \ref{lem:infinite_primes}, it is impossible to approximate the number of primes in $\mathbb{Z}[\sqrt{2}]$ within a certain norm region. However, as the number of $k$-norm curves within any $\NR(r^2)$ is finite, we now shift our goal to estimate the number of $k$-norm curves that contain primes in $\mathbb{Z}[\sqrt{2}]$ and reside within the norm region $\NR(r^2)$. For convenience, for any elements in $\mathbb{Z}[\sqrt{2}]$ that lie on the same $k$-norm curve, we say that they are in the same \textit{family}.

\begin{thm} \label{thm:primes_in_NR}
The number of families of primes in $\mathbb{Z}[\sqrt{2}]$ within $\NR(r^2)$ is asymptotic to $$\frac{r^2}{2\log r}.$$
\end{thm}

\begin{proof} According to Theorem \ref{thm:sqrt2primes} \ref{thm:form3}, the number of primes in $\mathbb{Z}[\sqrt{2}]$ within $\NR(r^2)$ that lie on the positive real axis is the same as the number of real primes $\equiv 3,5$ (mod $8$) that are less than $r$. Hence, by the PNT and Theorem \ref{thm:dirichlet_density}, there are asymptotically
\begin{equation}  2 \cdot\frac{r}{\varphi(8)\log r} =   \frac{r}{2\log r} \label{eq:primes3,5mod8} \end{equation} such primes. Since, for each prime $p$ on positive real axis, $-p$ and $p$ are in the same family $\NC(p^2)$, these primes then supply asymptotically $r/(2\log r)$ prime families within $\NR(r^2)$ when considering the entire real axis.

Similarly, there are asymptotically \begin{equation} 2 \cdot 2 \cdot \frac{r^2}{\varphi(8)\log r^2} = \frac{r^2}{2\log r} \label{eq:primes1,7mod8}\end{equation} real primes $p \equiv 1,7$ (mod $8$) from $-r^2$ to $r^2$. From Theorem \ref{thm:sqrt2primes} \ref{thm:form2}, each norm $p$ supplies a unique family $\NC(p)$.

Combining \eqref{eq:primes3,5mod8}, \eqref{eq:primes1,7mod8}, and another family of primes containing $\pm \sqrt{2}$, $\NC(-2)$, there are asymptotically \begin{equation*} \frac{r^2}{2\log r} \end{equation*} distinct families of primes within the region $\NR(r^2)$.
\end{proof}

The above analysis also confirms our observation that primes in $\mathbb{Z}[\sqrt{2}$] are most dense near the asymptotes $y = \pm x/ \sqrt{2}$, since the region $\NR(r^2)$ for any $r$ each straddles the asymptotes. The distribution of primes in $\mathbb{Z}[\sqrt{2}]$ differs from the Gaussian primes which are most dense near the origin. This fact may have implications towards how to construct the longest walk possible in $\mathbb{Z}[\sqrt{2}]$.

Now that we have an estimate on the number of families of primes within $\NR(r^2)$, we want to determine the total number of families of $\mathbb{Z}[\sqrt{2}]$ integers within this region to make a statement about the probability of encountering a family of primes. Note that not every $\NC(k)$ such that $|k| \le r^2$ contains element of $\mathbb{Z}[\sqrt{2}]$ because integer $k$ 
may not be expressible as $x^2-2y^2$ for some $x,y \in \mathbb{Z}$. 
Thus, to count the number of families of $\mathbb{Z}[\sqrt{2}]$ integers within $\NR(r^2)$, we refer to the following variation of a well-known theorem of Landau \cite{SS}. Note that a more general result can be found in Bernays' 1912 thesis \cite{B}.

\begin{thm} \label{thm:Bernaystheorem}
For any integer $m \neq -k^2$. Let $B_m(n)$ be the number of positive integers less than or equal to $n$ that can be expressed as $f(x,y) = x^2+my^2$ where $x,y \in \mathbb{Z}$. Then
$$B_m(n) \sim \frac{b_mn}{\sqrt{\log n}},$$
where $b_m$ is some positive constant.
\end{thm}

To apply Theorem \ref{thm:Bernaystheorem}, we take $n=r^2$ and $f(x,y) = x^2-2y^2$ and verify that $ m= -2$ is not a negative square. Then we can find an asymptotic estimate for the number of positive integers $k < r^2$ such that $$k = x^2-2y^2 = N(x+y\sqrt{2})$$ for some $x+y\sqrt{2} \in \mathbb{Z}[\sqrt{2}]$. In other words, we have an estimate for the number of families of integers in $\mathbb{Z}[\sqrt{2}]$ with positive norms that reside in $\NR(r^2)$. We determine the number of families of integers with negative norms as follows.

\begin{lem}\label{lem:sign_invariant}
There exist $a_1, b_1 \in \mathbb{Z}$ such that $a_1^2 - 2b_1^2 = k$ if and only if there exist $a_2, b_2 \in \mathbb{Z}$ such that $a_2^2 - 2b_2^2 = -k$.
\end{lem}

\begin{proof}
This can be proved using the fact that the norm in $\mathbb{Z}[\sqrt{2}]$ is multiplicative. Since $N(1+\sqrt{2}) = -1$, we simply have $$N(a_1 + b_1\sqrt{2})N(1+\sqrt{2}) = -k = N((a_1 + b_1\sqrt{2})(1+\sqrt{2})).$$
\end{proof}

In geometric terms, Lemma \ref{lem:sign_invariant} tells us that if we have an integer that lies on the norm curve $\NC(k)$, then there also exists an integer that lies on the related norm curve $\NC(-k)$, and vice versa. Therefore, for any $r \in \mathbb{N}$, the number of families of $\mathbb{Z}[\sqrt{2}]$ integers within $\NR(r^2)$ is asymptotic to \begin{equation}\label{eq:number_families}
    2\cdot\frac{b_{-2}r^2}{\sqrt{\log r^2}},
\end{equation}
where $b_{-2}$ is a constant as stated in Theorem \ref{thm:Bernaystheorem}.

Then we can divide our estimate of the number of families of primes in $\NR(r^2)$ from Theorem \ref{thm:primes_in_NR} by the total number of families of $\mathbb{Z}[\sqrt{2}]$ integers in $\NR(r^2)$ in \eqref{eq:number_families} to determine the probability of encountering a family of primes in $\NR(r^2)$. Thus, our work culminates in the following estimate.

\begin{thm} \label{thm:prob_prime_in_NR}
The probability that a $\mathbb{Z}[\sqrt{2}]$ integer $z \in \NR(r^2)$ is in a family of prime is asymptotic to $$\frac{1}{2b_{-2}\sqrt{2\log r}}.$$
\end{thm}

It is worth noting that the analogous probabilities for $\mathbb{Z}$ and $\mathbb{Z}[i]$ are asymptotic to $1/\log r$ and $2/(\pi\log r)$, respectively, which indicates a greater density of primes for $\mathbb{Z}[\sqrt{2}]$ within its respective norm-region.

\section{Proof of main results}\label{sec:proof_of_main_results}
In the previous section, we have shown that, within a fixed norm region, the number of primes in $\mathbb{Z}[\sqrt{2}]$ is infinite, yet the number of families of primes, i.e., the number of $k$-norm curves containing primes, is finite. While these curves have infinitely many primes, their exponential growth renders them increasingly and negligibly sparse as our walk of linear growth rate progresses. This observation leads to our main theorem in this section, which states that for any positive integers $k$ and $r$, there is no unbounded walk in the first quadrant of $\mathbb{Z}[\sqrt{2}]$ that steps on the primes, has step length bounded by $k$, and has distance from $y = x/\sqrt{2}$ bounded by $r$.

We now introduce the following lemmas that are needed for the proof of the main theorem.

\begin{lem}\label{lem:allprimesinNC} For any norm curve $\NC(k)$ containing some prime $P=x+y\sqrt{2}$, another integer $P'$ on the same curve is also a prime if and only if $P'$ is of the form $\pm(x\pm y\sqrt{2})(1+\sqrt{2})^{2m}$ for some $m \in \mathbb{Z}$.
\end{lem}
\begin{proof}
($\leftarrow$) If $P'$ can be written as $\pm(x+ y\sqrt{2})(1+\sqrt{2})^{2m}$, then $P'$ is a prime as it is an associate of $x+y\sqrt{2}$. Thus, we only need to show that any conjugate of a prime is also a prime, i.e., $x-y\sqrt{2}$ is a prime. Let us denote the conjugate of $a+b\sqrt{2}$ by $\overline{a+b\sqrt{2}}$. Notice that \begin{equation}\label{eq:conjugate}
     \overline{(a+b\sqrt{2})(c+d\sqrt{2})}= (\overline{a+b\sqrt{2}})(\overline{c+d\sqrt{2}}).
\end{equation}
Hence, the conjugate of a product of elements in $\mathbb{Z}[\sqrt{2}]$ is the product of the conjugate of each element.

From Theorem \ref{thm:sqrt2primes}, we know that any prime $P$ can be written as $\pm Q(1+\sqrt{2})^m$, $m \in \mathbb{Z}$, where $Q$ is either $\sqrt{2}$, a real prime $\equiv 3,5$ (mod $8$), or $x+y\sqrt{2}$ where $x^2-2y^2$ is a real prime $\equiv 1,7$ (mod $8$). Then we can easily check that $\overline{Q}$ is also a prime for each case. Hence, by \eqref{eq:conjugate}, $$\overline{P} = \pm \overline{Q(1+\sqrt{2})^{k}} = \pm\overline{Q}\overline{(1+\sqrt{2})^{k}} = \pm\overline{Q}(-1)^k{(1+\sqrt{2})^{-k}},$$ which is an associate of the prime $\overline{Q}$. Therefore, $\overline{P}$ is also a prime.

($\rightarrow$) Suppose that $P' = x'+y'\sqrt{2}$ is another prime on the same curve but not of the form $\pm(x\pm y\sqrt{2})(1+\sqrt{2})^{2m}$. Since both $P$ and $P'$ are on the same norm curve, we have \begin{equation}\label{eq:primesamenorm}
N(P') = (x'+y'\sqrt{2})(x'-y'\sqrt{2}) = (x+y\sqrt{2})(x-y\sqrt{2}) = N(P).    
\end{equation} By definition of primes in a UFD, both $x+y\sqrt{2}$ and $x-y\sqrt{2}$ are divisible only by its associates and the units of the UFD, so $x'+y'\sqrt{2}$ is either a unit or an associate of $x^2-2y^2$. Either case implies that one of $x'+y'\sqrt{2}$ and $x'-y'\sqrt{2}$ must be a unit.  Thus, $|N(P')| = |N(x'+y'\sqrt{2})| = |N(x'-y'\sqrt{2})| = 1$, a contradiction because $|N(P)| \neq 1$. 
\end{proof}

Using the same argument as in the forward direction proof of Lemma \ref{lem:allprimesinNC}, we can show that if $P = x+y\sqrt{2}$ is a prime and $Q = x'+y'\sqrt{2}$ is a $\mathbb{Z}[\sqrt{2}]$ integer with the same norm, then $Q$ is also a prime. This can be shown by considering \eqref{eq:primesamenorm} and noticing that neither $x'+y'\sqrt{2}$ nor $x'-y'\sqrt{2}$ can be a unit. Then each of them must be an associate of $P$ or of the conjugate of $P$. Therefore, \textit{if there is a prime on a $k$-norm curve, every integer on that curve is also a prime.}

\begin{lem}\label{lem:two_primes} If $P= x+y\sqrt{2}$ is a prime in $\mathbb{Z}[\sqrt{2}]$ where $x, y \in \mathbb{N}$, then there is at most one prime between $P$ and $P' = P(1+\sqrt{2})^2 = (3x+4y) +(2x+3y)\sqrt{2}$, which are both on the same norm curve.
\end{lem}

\begin{proof}
Consider the norm curve $\NC(x^2-2y^2)$ and, without loss of generality, suppose that $x^2-2y^2 > 0$, i.e., we examine the black curves in Figure \ref{fig:position_primes}. A similar proof can also be achieved for when $x^2-2y^2 < 0$, i.e., the blue curves in the same figure.

We first let $S = \{x+y\sqrt{2} \in \mathbb{Z}[\sqrt{2}] | x,y >0 \} \cap \NC(x^2-2y^2)$, which is the set of all $\mathbb{Z}[\sqrt{2}]$ integers on the norm curve $\NC(x^2-2y^2)$ in the first quadrant. There must be a prime $P_0 = x_0 +y_0\sqrt{2}$ in the set $S$ where $x_0, y_0 >0$ with the smallest Euclidean norm as the Euclidean norm of each element in $S$ is of the form $\sqrt{M}$, where $M \in \mathbb{N}$.

\begin{figure}[ht]
    \centering
    \includegraphics[width = 0.50\textwidth]{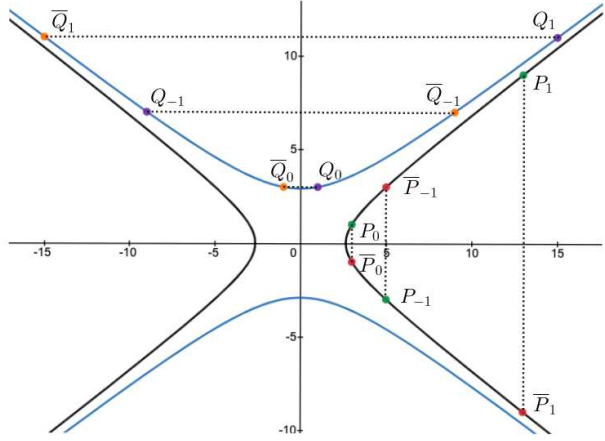}
    \caption{Positions of primes on the same norm curves.}
    \label{fig:position_primes}
\end{figure}

For convenience, we let $P_k$ denote the prime $x_k +y_k\sqrt{2} = (x_0 +y_0\sqrt{2})(1+\sqrt{2})^{2k}$ and $\overline{P}_k$ denote its conjugate. Then, from Lemma \ref{lem:allprimesinNC}, each prime on the curve $\NC(x^2-2y^2)$ is either $P_k$, $-P_k$, $\overline{P}_k$, or $-\overline{P}_k$ for some $k \in \mathbb{Z}$. We also note that since $\NC(x^2-2y^2)$ is symmetric about the $y$-axis, the $x$-coordinate of all primes on the right curve of $\NC(x^2-2y^2)$ must be positive.

We now show that all $P_k$'s and $\overline{P}_k$'s are on the right curve of $\NC(x^2-2y^2)$, hence all $-P_k$'s and $-\overline{P}_k$'s are on the left. For any $P_k = x_k +y_k \sqrt{2}$ where $x_k, y_k >0$, we have that $$P_{k+1} = P_k(1+\sqrt{2})^2 = (3x_k+4y_k)+(2x_k+3y_k)\sqrt{2}.$$
Since $x_k, y_k >0$, we obtain that $3x_k+4y_k > x_k$ and $2x_k +3y_k>y_k$. Hence, $x_{k+1} > x_k > 0$ and $y_{k+1} > y_k > 0$ for any $k \in \mathbb{N}$. Therefore, all $P_k$'s for $k \in \mathbb{N}$ are on the right curve.

Next, we consider $P_{-k}$ where $k>0$. In particular, we have $$P_{-1} = x_{-1}+y_{-1}\sqrt{2} = P_0(1+\sqrt{2})^{-2}  = (3x_0-4y_0)+(3y_0-2x_0)\sqrt{2}.$$
Since we consider the norm curve $x^2-2y^2 >0$, we have $y_0 < x_0/\sqrt{2}$. Hence, $x_{-1} = 3x_0-4y_0>3x_0-2\sqrt{2}x_0 = (3-2\sqrt{2})x_0 >0$, so $P_{-1}$ is still on the right curve of $\NC(x^2-2y^2)$. If $P_{-1}$ is on the upper half-curve, we would have $x_{-1} > x_0$ and $y_{-1} > y_0$ because the graph is increasing and $P_0$ has the smallest Euclidean norm with $x_0, y_0 >0$. However, this is impossible because we would have $ y_{-1} = 3y_0 - 2x_0 > y_0$ implying $y_0 > x_0$. Hence, $y_{-1} < 0$, i.e., $P_{-1}$ is on the lower half-curve. Similarly, we can show that $x_{-k-1} > x_{-k}$ and $y_{-k-1} < y_{-k}$ for $k \ge 1$.

Therefore, all $P_k$'s where $k \in \mathbb{Z}$ lie on the right curve of $\NC(x^2-2y^2)$, implying that all $\overline{P}_k$'s lie on the right norm curve while $-P_k$'s and $-\overline{P}_k$'s lie on the left curve. In particular, $P_k$'s and $\overline{P}_{-k-1}$'s for $k \ge 0$ are all primes on the \textit{right upper} half-curve, and their locations are shown in Figure \ref{fig:position_primes}.

Lastly, we show that there is at most one prime between $P$ and $P'$. We note that $P$ is either $P_m$ or $\overline{P}_m$ for some $m \in \mathbb{Z}$. Since these two cases have similar proofs, we only show the prior one. Suppose that there are at least two primes between $P= P_m$ and $P' = P_{m+1}$, which are both on the upper half-curve, i.e., $m\ge 0$. Then they must be $\overline{P}_{-n}$ and $\overline{P}_{-n-1}$ for some $n > 0$. However, we know that $x_m < x_{-n}$, $y_m <-y_{-n}$,
$x_{m+1} = 3x_m + 4y_m$, and $y_{m+1} = 3y_m+2x_m$.
Hence, it follows that
$$ x_{-n-1} = 3x_{-n} - 4y_{-n} > 3x_m +4y_m = x_{m+1},$$
and 
$$-y_{-n-1} = -3y_{-n}+2x_{-n} > 3y_m+2x_m = y_{m+1}.$$

Therefore, $\overline{P}_{-n-1} =  x_{-n-1} - y_{-n-1}\sqrt{2}$ is not between $P_m$ and $P_{m+1}$, a contradiction. Thus, there can be at most one prime between $P$ and $P'$.
\end{proof}

We are now ready to prove the main theorem.

\begin{thm} \label{thm:main_theorem}
For any positive integers $k$ and $r$, there is no unbounded walk in the first quadrant of $\mathbb{Z}[\sqrt{2}]$ that steps on the primes, has step length bounded by $k$, and has distance from $y = x/\sqrt{2}$ bounded by $r$.
\end{thm}

\begin{proof} Our goal is to construct a bounded region where every prime within it is of at most distance $r$ from the line $y = x/\sqrt{2}$ and then show that the number of primes in that region is clearly less than the minimum number of steps needed to walk across it.

Assume that any possible walk to infinity remains within some bounded distance $r$ from $y = x/\sqrt{2}$ and that we take steps of length at most $k$. Let $S=(p,q)$ be some sufficiently large prime (in terms of its Euclidean norm) that belongs to the walk, and, without loss of generality, $S$ is below the asymptote. We then construct the line perpendicular to $y = x/\sqrt{2}$ passing through $S$. Let this line intersect $y = x/\sqrt{2}$ at $P=(x,y)$ as in Figure \ref{fig:constructregion}. Then we construct $A$ and $B$ on the line perpendicular to the asymptote such that $AP = BP = r$. By construction, $S$ is between $A$ and $P$.

\begin{figure}[ht]
    \centering
    \includegraphics[width = 0.50\textwidth]{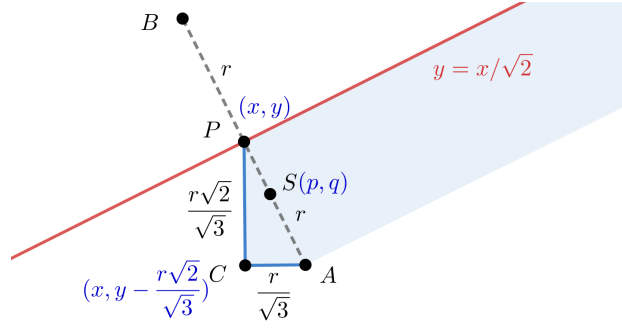}
    \caption{Constructing the desired bounded area.}
    \label{fig:constructregion}
\end{figure}

Now, we construct $C$ to be the point where the vertical line passing through $P$ intersects the horizontal line passing through $A$. It follows that $CA = \frac{r}{\sqrt{3}}$ and $CP = \frac{r\sqrt{2}}{\sqrt{3}}$. Hence, $C = \left(x, y- \frac{r\sqrt{2}}{\sqrt{3}}\right)$. Note that now, for any point in the shaded area in Figure \ref{fig:constructregion}, the point is always above and to the right of $C$.

We now consider any prime $K = (m,n)$ in the shaded area and $K'$ be the coordinates that represent the prime $(m+n\sqrt{2})(1+\sqrt{2})^2$, i.e., $K' = (3m+4n, 2m+3n)$. Then $$d(K,K') = \sqrt{(2m+4n)^2+(2m+2n)^2}.$$
On the other hand, if we define $C'$ similarly to $K'$, i.e., $$C' = \left(3x + 4\left(y- \frac{r\sqrt{2}}{\sqrt{3}}\right), 2x+ 3\left(y- \frac{r\sqrt{2}}{\sqrt{3}}\right)\right),$$we then have \begin{equation}
    \label{eq:d(C,C')} d(C,C') = \sqrt{\left(2x+4\left(y- \frac{r\sqrt{2}}{\sqrt{3}}\right)\right)^2+\left(2x+2\left(y- \frac{r\sqrt{2}}{\sqrt{3}}\right)\right)^2}.\end{equation}
Since $m > x$ and $n> y- \frac{r\sqrt{2}}{\sqrt{3}}$ because $K$ is above and to the right of $C$, we have that $d(K,K') > d(C,C')$, for any $K$ in the shaded region in Figure \ref{fig:constructregion}.

\begin{figure}[ht]
    \centering
    \includegraphics[width = 0.45\textwidth]{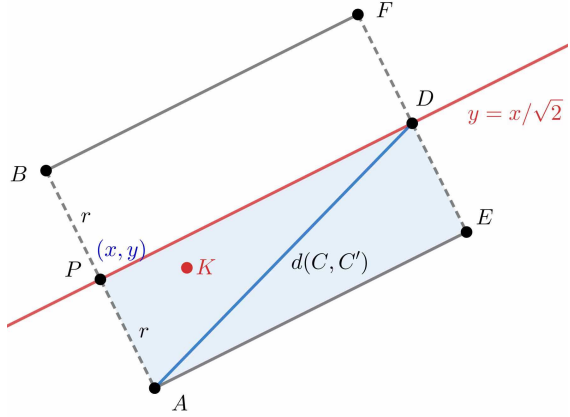}
    \caption{Constructing the bounded region $APDE$.}
    \label{fig:APDE}
\end{figure}

Next, we construct $D$ on the asymptote such that $AD = d(C,C')$ and let $E$ be the point such that $\overline{ED}//\overline{AP}$ and $ED = r$ as in Figure \ref{fig:APDE}. Then, in the rectangle $APDE$, the greatest distance between any 2 points within the region is $d(C,C')$. As we have observed that $d(K,K') > d(C,C')$ for any $K$ in the shaded region in Figure \ref{fig:constructregion}, the same is true for any $K$ in the rectangle $APDE$. Equivalently, $d(K,K') > d(C,C')$ for any prime $K = m+n\sqrt{2}$ in $APDE$.

In other words, within the rectangle $APDE$, if we step on some prime $K$, the prime $K'$ will not lie inside $APDE$. This means \textit{there are at most 2 primes on the same norm curve within $APDE$}, namely $K$ and the prime between $K$ and $K'$ if there is any, according to Lemma \ref{lem:two_primes}.

With the above fact, let us now estimate the minimum number of steps needed to walk across $ABFE$, where $F$ is the reflection of $E$ about the asymptote as in Figure \ref{fig:APDE}, and the number of possible primes we can step on in the same area. Let $g(x)$ be the vertical gap between the asymptote $y = x/\sqrt{2}$ and the curve $x^2-2y^2 =c$ in the first quadrant, which can be rewritten as $y = {\sqrt{x^2-c}}/{\sqrt{2}}$. Then $g(x)$ can be computed as
$$g(x) = \frac{x-\sqrt{x^2-c}}{\sqrt{2}} = \frac{c}{\sqrt{2}(x+\sqrt{x^2-c})} \ge \frac{c}{2\sqrt{2} x}.$$

\begin{figure}[ht]
    \centering
    \includegraphics[width = 0.45\textwidth]{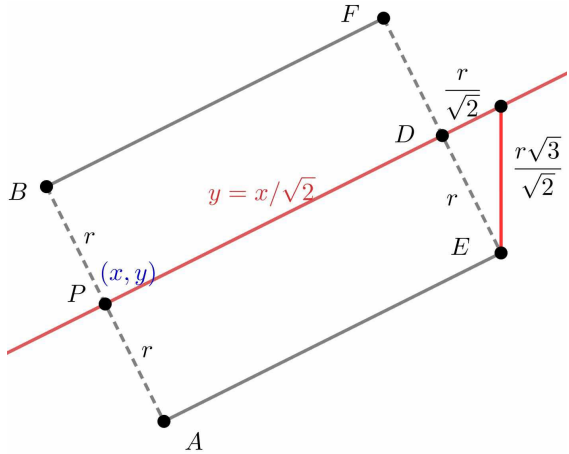}
    \caption{The vertical length from $E$ to the asymptote $y=x/\sqrt{2}$.}
    \label{fig:verticallength}
\end{figure}

To find the largest $c$ such that the curve $y = \sqrt{x^2-c}/{\sqrt{2}}$ stays within $APDE$, we need $g(x)$ not to exceed the vertical length from $E$ to the asympote, which is $r\sqrt{3}/\sqrt{2}$ as shown in Figure \ref{fig:verticallength}. Thus $$ \frac{c}{2\sqrt{2}x} \le g(x) \le \frac{r\sqrt{3}}{\sqrt{2}},$$
so $c \le 2\sqrt{3} rx$. 

However, not all of the norm curves within the norm region $\NR(2\sqrt{3}rx)$ contain primes. From Theorem \ref{thm:primes_in_NR}, within $\NR(2\sqrt{3}rx)$, the number of norm curves that contain primes is asymptotic to
$$\frac{2\sqrt{3}rx}{2\log\sqrt{2\sqrt{3}rx}} = O\left(\frac{x}{\log x}\right).$$ 
In particular, as we have observed that there are at most 2 primes on the same norm curve within $APDE$ and hence $ABFE$ by symmetry, the number of primes we can step on is at most twice of the estimate above, which is still $O\left(x/{\log x}\right)$.

On the other hand, we need to take as few as $\lceil PD/k \rceil$ steps to go from the front $\overline{AB}$ to the front $\overline{EF}$. Thus, from \eqref{eq:d(C,C')}, the number of steps needed is bounded below by
\begin{align*}
    \frac{PD}{k}  = \frac{\sqrt{AD^2-AP^2}}{k}
    &= \frac{\sqrt{d(C,C')^2-r^2}}{k}\\
    &= \frac{\sqrt{\left(2x+4\left(y- \frac{r\sqrt{2}}{\sqrt{3}}\right)\right)^2+\left(2x+2\left(y- \frac{r\sqrt{2}}{\sqrt{3}}\right)\right)^2 -r^2}}{k}.
\end{align*}
Note that $P=(x,y)$ is on $y = x/\sqrt{2}$, so replacing $y$ with $x/\sqrt{2}$ gives us that the number of steps needed is $\Omega(x)$, whereas the number of primes within $ABFE$ is $O(x/\log x)$. Therefore, for a sufficiently large $S$ hence $x$, the number of primes we can step on within $ABFE$ is much less than the number of steps needed to walk across the same region.
\end{proof}

By the symmetry of primes in $\mathbb{Z}[\sqrt{2}]$, the theorem is also true for the third quadrant and the second and fourth quadrants when considering the asymptote $y=-x/\sqrt{2}$ instead.

\section{Visualizing prime walks}\label{sec:visualizing_prime_walks}
In this section, we present some evidence of why every prime walk in $\mathbb{Z}[\sqrt{2}]$ with steps of bounded length is likely to stay close to the asymptote. We construct a random walk on primes in $\mathbb{Z}[\sqrt{2}]$ with steps of length at most $k$ by starting at $\sqrt{2}$ and searching randomly for the next prime within distance $k$ from the previous one. We also require the next prime has a larger Euclidean norm than the previous one to let the walk always spread further away from the origin. Figure \ref{fig:random_walk_at_origin} shows a collection of such random walks with step size $\sqrt{8}$ starting from $\sqrt{2}$. It is interesting to discover that nearly all of the longest random walks are close to the asymptote $y= \pm x/\sqrt{2}$. When the starting point is chosen randomly and not necessarily $\sqrt{2}$, the resulting random walks also have the same property. In Figure \ref{fig:randome_walk_not_at_origin}, we have a collection of random walks with step size $\sqrt{8}$ starting from $13+15\sqrt{2}$. Again, the longest walk is very close to the asymptote $y= \pm x/\sqrt{2}$.

\begin{figure}[ht]
    \centering
    \includegraphics[width=0.5\textwidth]{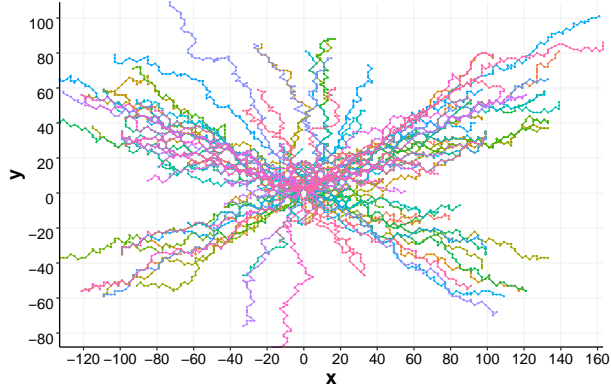}
    \caption{Random walks on primes with $k=\sqrt{8}$ starting at $\sqrt{2}$.}
    \label{fig:random_walk_at_origin}
\end{figure}

\begin{figure}[ht]
    \centering
    \includegraphics[width=0.5\textwidth]{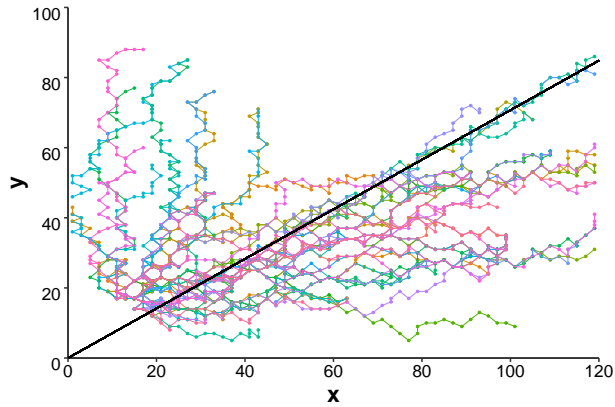}
    \caption{Random walks on primes with $k=\sqrt{8}$ starting at $13+15\sqrt{2}$ with the asymptote $y=x/\sqrt{2}$ (black).}
     \label{fig:randome_walk_not_at_origin}
\end{figure}

\begin{figure}[ht]
    \centering
    \includegraphics[width=0.4\textwidth]{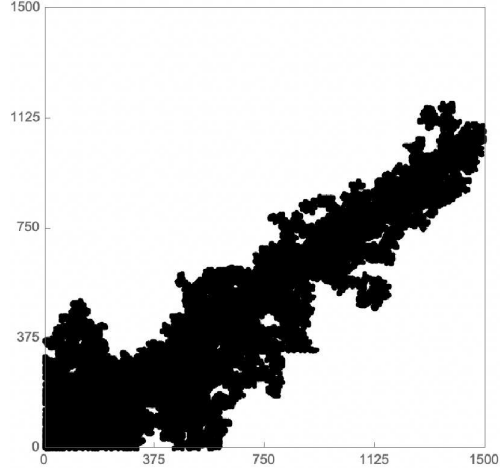}
    \caption{All prime walks with $k = \sqrt{8}$ from the origin in the first quadrant.}
    \label{fig:primewalksqrt8}
\end{figure}

We now describe the algorithm used to visualize all possible prime walks for step size $k=\sqrt{8}$ starting from the origin. The algorithm consists of four main steps. The first step is to identify all primes in a disk of radius $n$ for some large $n$ by checking whether each element in the disk satisfies Theorem \ref{thm:sqrt2primes}. Then, for each prime $p$, we find its $k$-neighbors which are primes within the distance $k$ from $p$. Next, we form the road network connecting all primes and their $k$-neighbors, and find the connected component of $\sqrt{2}$, which is the prime closest to the origin. Finally, we can calculate all possible prime walks (ideally, if not limited by the running time) with steps of length at most $k$ starting from the origin.

When $k$ is $\sqrt{8}$, Figure \ref{fig:primewalksqrt8} shows that the prime walks can go as far as $x =1500$. Note that in reality the walks do not end there, but due to the limitation of our algorithm, the running time increases dramatically as the walks progress, so we need to stop the calculation after a certain point. Still, there is a clear trend that the longest walk is near the asymptote $y = x/\sqrt{2}$, so the condition on the bounded distance from $y=x/\sqrt{2}$ in Theorem \ref{thm:main_theorem} may not be necessary. 

\begin{conj}\label{conj:main}
For any positive integers $k$, there is no unbounded walk in $\mathbb{Z}[\sqrt{2}]$ that steps only on the primes and has step length bounded by $k$.
\end{conj}

\section{Conclusion}\label{sec:conclusion}
By estimating the number of primes near the asymptote $y=x/\sqrt{2}$, we have shown that there is no prime walk to infinity with steps of bounded length in $\mathbb{Z}[\sqrt{2}]$ if every prime walk in $\mathbb{Z}[\sqrt{2}]$ is guaranteed to stay within some bounded distance from $y= x/\sqrt{2}$. With several moat computations, we have demonstrated that the longest prime walks tend to cluster near the asymptote, which leads to Conjecture \ref{conj:main}. One direction for future work is to prove Conjecture \ref{conj:main} and generalize this result to other quadratic integer rings that are UFDs, such as $\mathbb{Z}[\sqrt{3}]$, $\mathbb{Z}[\frac{1 + \sqrt{5}}{2}]$, or $\mathbb{Z}[\sqrt{7}]$. Figures \ref{fig:primeinQroot2} and \ref{fig:primeinQroot3} show that primes in $\mathbb{Z}[\sqrt{3}]$ also cluster near the asymptote $x^2-3y^2=0$ similar to the case of $\mathbb{Z}[\sqrt{2}]$, and hence suggest that there should exist no prime walk to infinity with steps of bounded length in $\mathbb{Z}[\sqrt{3}]$. Another direction is to study prime walks in the quadratic integer rings that are not UFDs, such as $\mathbb{Z}[\sqrt{-5}]$, since our proof of the main theorem requires the ring to be a UFD.

\begin{figure}[ht]
  \begin{subfigure}[t]{0.48\textwidth}
  \centering
    \includegraphics[width=0.80\textwidth]{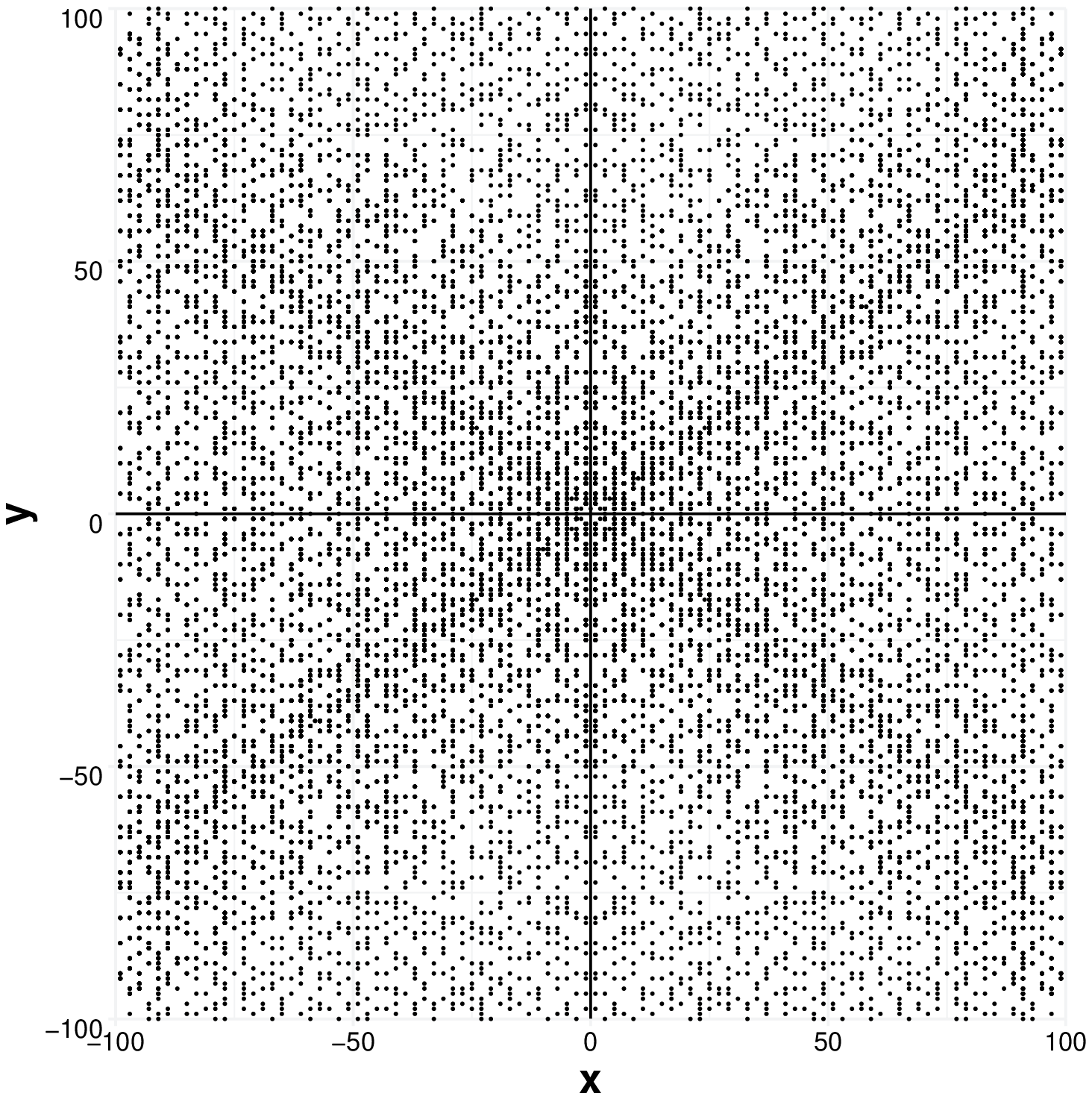}
    \caption{}
    \label{fig:primeinQroot2}
  \end{subfigure}
  \begin{subfigure}[t]{0.48\textwidth}
  \centering
    \includegraphics[width=0.80\textwidth]{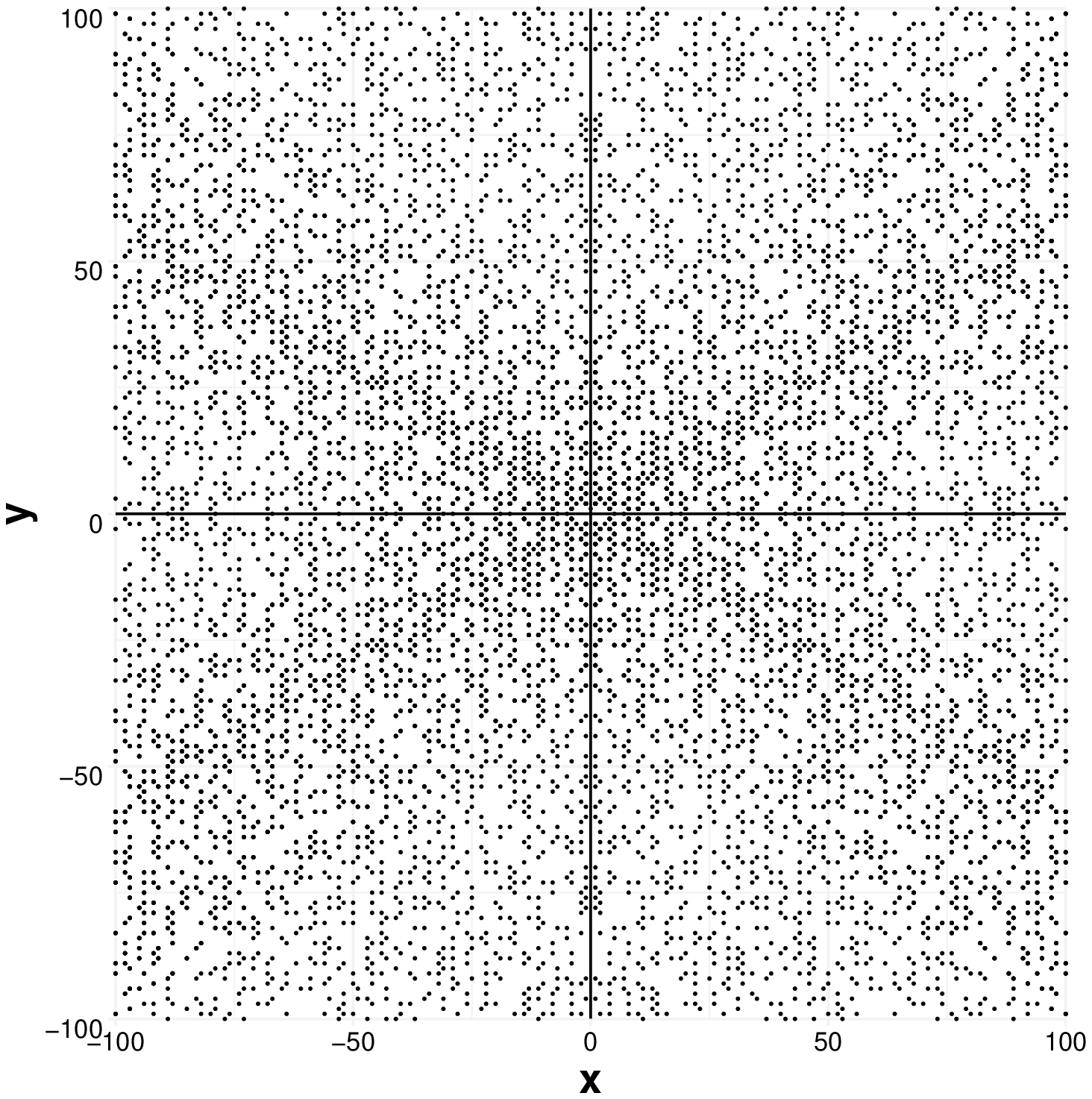}
    \caption{}
    \label{fig:primeinQroot3}
  \end{subfigure}
  \caption{Primes in $\mathbb{Z}[\sqrt{2}]$ (Figure \ref{fig:primeinQroot2}) and $\mathbb{Z}[\sqrt{3}]$ (Figure \ref{fig:primeinQroot3}).}
\end{figure}

\section{Acknowledgments}

We thank other members of the collaborative team at the Polymath REU program for productive discussions and useful comments. We would like to thank our colleagues at the Young Mathematicians Conference in August 2020 for useful comments. We would also like to extend our gratitude to the editor and referees from the \textit{Journal of Integer Sequences} for careful reading and comments on the manuscript.

\appendix

\bigskip
\hrule
\bigskip

\noindent 2010 {\it Mathematics Subject Classification}:~Primary 11A41; Secondary 11B05.

\noindent \textit{Keywords}:~walking to infinity, prime walk, $\mathbb{Z}[\sqrt{2}]$, real quadratic integer ring.

\bigskip
\hrule
\bigskip

\end{document}